\documentclass[12pt, reqno]{amsart}
\usepackage{amsmath, amsthm, amscd, amsfonts, amssymb, graphicx, color}
\usepackage[bookmarksnumbered, colorlinks, plainpages]{hyperref}
\hypersetup{colorlinks=true,linkcolor=red, anchorcolor=green, citecolor=cyan, urlcolor=red, filecolor=magenta, pdftoolbar=true}

\newtheorem{theorem}{Theorem}[section]
\newtheorem{lemma}[theorem]{Lemma}
\newtheorem{proposition}[theorem]{Proposition}
\newtheorem{corollary}[theorem]{Corollary}
\theoremstyle{definition}

\newtheorem{example}[theorem]{Example}

\theoremstyle{remark}
\newtheorem{remark}[theorem]{\textbf{Remark}}
\numberwithin{equation}{section}

\begin{document}

\title[Bounds of numerical radius of bounded linear operators]{Bounds of numerical radius of bounded linear operators using $t$-Aluthge transform}


\author{Santanu Bag, Pintu Bhunia and Kallol Paul}

\address{(Bag) Department of Mathematics, Vivekananda College For Women, Barisha, Kolkata 700008, India}
\email{santanumath84@gmail.com}

\address{(Bhunia) Department of Mathematics, Jadavpur University, Kolkata 700032, West Bengal, India}
\email{pintubhunia5206@gmail.com}

\address{(Paul) Department of Mathematics, Jadavpur University, Kolkata 700032, West Bengal, India}
\email{kalloldada@gmail.com}









\thanks{The authors would like to thank the referees for their insightful suggestions that helped us to improve this article. The research of Pintu Bhunia would like to thank UGC, Govt. of India for the financial support in the form of junior research fellowship. Prof. Kallol Paul would like to thank RUSA 2.0, Jadavpur University for  partial support} 

\begin{abstract}
       	We develop a number of inequalities to obtain bounds for the numerical radius of a bounded linear operator defined on a complex Hilbert space using the properties of $t$-Aluthge transform. We show that the bounds obtained are sharper than the existing bounds.
\end{abstract}

\maketitle



\section{Introduction}

       \noindent Let $B(\mathbb{H})$ denote the $C^*$-algebra of all bounded linear operators defined on a complex Hilbert space $\mathbb{H}$. For $T\in B(\mathbb{H})$, the numerical range of $T$ is defined as $$W(T)=\{\langle Tx,x\rangle:x \in \mathbb{H}, \|x\|=1 \}.$$ The numerical radius, $w(T)$, is defined as the radius of the smallest circle with centered at the origin and containing the numerical range, i.e., $$w(T) =  \sup\{|\lambda|:\lambda \in W(T)\}.$$ The Crawford number of $T$ is defined as $$m(T)=\inf\{|\lambda|:\lambda \in W(T)\}.$$  The Cartesian decomposition of $T$ is given by $T=Re(T)+{\rm i}~Im(T)$, where $Re(T)=\frac{T+T^*}{2}$ and $Im(T)=\frac{T-T^*}{2{\rm i}}$. The spectral radius of $T$ is defined as $$r(T)=\sup\{|\lambda|:\lambda \in \sigma(T)\}$$ where $\sigma(T)$ is the collection of all spectral values of $T$. It is well-known that $w(T)$ defines a norm on $B(\mathbb{H})$,  which is equivalent to the operator norm $\|.\|$,  satisfying  the following inequality  
       \[ \frac{1}{2}\|T\|\leq w(T)\leq \|T\|. \]
       The first inequality becomes an equality if $T^2=0$ and the second inequality becomes an equality if $T$ is normal. 
       Various mathematicians \cite{BBP,CM, GW, FK2,FK1,PB,Y,ZFW} have studied and improved on the numerical radius inequality over the years using different techniques. One of the substantive improvement of the upper bound of the numerical radius was done  by Kittaneh \cite[Th. 1]{FK1}, in which he proved that 
       \begin{eqnarray}\label {number1}
       w(T) & \leq & \frac{1}{2}\left(\|T\|+\|T^2\|^{\frac{1}{2}}\right).
       \end{eqnarray}
       Using the Cartesian decomposition of an operator, Kittaneh \cite[Th. 1]{FK2} also proved that 
       \begin{eqnarray}\label {number2}
       \frac{1}{4}\|T^* T+TT^*\| &\leq& w^2(T)\leq \frac{1}{2}\|T^* T+TT^*\|.
       \end{eqnarray}
       For $T\in B(\mathbb{H})$, the Aluthge  transform  of $T$, denoted as $\widetilde{T}$, is defined as \[\widetilde{T}=|T|^{\frac{1}{2}}U|T|^{\frac{1}{2}},\] where $|T|=(T^{*}T)^{\frac{1}{2}}$ and $U$ is the partial isometry associated with the polar decomposition of $T$ and so $T=U|T|,$ $ \ker T=\ker U.$ It follows easily from the definition of  $\widetilde{T}$ that  $\|\widetilde{T}\| \leq \|T\|$ and  $r(\widetilde{T})= r(T)$, also  $w(\widetilde{T}) \leq w(T)$ (see \cite{JKP}). Okubo \cite{O} generalized the Aluthge transform, known as the $t$-Aluthge transform as follows: \\
       For $t \in [0,1]$, the $t$-Aluthge transform is defined by,
       \[\widetilde{T_t}=|T|^{t}U|T|^{1-t}.\]
       Here, $|T|^0$ is defined as $U^{*}U$. In particular, $\widetilde{T_0}=U^{*}U^2|T|$, $\widetilde{T_1}=|T|UU^{*}U=|T|U$, $\widetilde{T_{\frac{1}{2}}}=|T|^{\frac{1}{2}}U|T|^{\frac{1}{2}}=\widetilde{T}$ (the Aluthge transform of $T$).\\
       Using the Aluthge transform, Yamazaki in \cite[Th. 2.1]{Y} proved that if  $T\in B(\mathbb{H})$, then 
       \begin{eqnarray} \label {number3}
       w(T) &\leq & \frac{1}{2}\left(\|T\|+w(\widetilde{T})\right).
       \end{eqnarray}
       He also proved that this inequality is better than inequality (\ref{number1}) obtained by Kittaneh \cite[Th. 1]{FK1}. 
       Abu-Omar and Kittaneh  \cite[Th. 3.2]{AF2} improved on inequality (\ref{number3}) using $t$-Aluthge transform to prove that 
       \begin{eqnarray} \label {number4}
       w(T) &\leq & \frac{1}{2}\big(\|T\|+\min_{t\in [0,1]}w(\widetilde{T_t})\big).
       \end{eqnarray}
       Clearly  inequality (\ref{number4}) is sharper  than (\ref{number3}) and hence (\ref{number1}).  We observe that inequalities  (\ref{number1}) and (\ref{number2})  as well as   (\ref{number4}) and (\ref{number2}) are not comparable, in general. In this paper, we develop a number of inequalities using the properties  of $t$-Aluthge transform. We show that our inequalities improve inequalities  (\ref{number1}), (\ref{number2}) and (\ref{number3}).  We also obtain an upper bound for the numerical radius and show by an example that the bound is better than that obtained in inequality (\ref{number4}) for some certain operators. 
        
\section{Main results}
We begin this section with two  notations $H_{\theta}$  and $K_{\theta}$, defined as follows:\\
For $T \in B(\mathbb{H})$ and $\theta \in \mathbb{R},$  $H_{\theta}:=\textit{Re}(e^{{\rm i} \theta}T)$ and $K_{\theta}:=\textit{Im}(e^{{\rm i} \theta}T)$.  

The following lemma (see \cite{Y}) will be used repeatedly to reach our goal in this present article.

\begin{lemma}\label{lemma:htheta}
	Let $T\in B(\mathbb{H}).$ Then
	\[w(T)=\sup_{\theta \in \mathbb{R}}\|H_\theta\|=\sup_{\theta \in \mathbb{R}}\|\textit{Re}(e^{{\rm i} \theta}T)\|.\]
	Replacing $T$ by $\rm iT$ in the above equation, we have
	\[w(T)=\sup_{\theta \in \mathbb{R}}\|K_\theta\|=\sup_{\theta \in \mathbb{R}}\|\textit{Im}(e^{{\rm i} \theta}T)\|.\]
\end{lemma}

We next prove the following proposition which states that $ T^2=0$ and $\widetilde{T}_t=0$, for any $t\in [0,1]$ are equivalent. To achieve it, we need the Heinz inequality (see \cite {H}) given below.
\begin{lemma}\cite{H} \label{Heinz}
	Let $A,B,X\in B(\mathbb{H})$ where $A$ and $B$ be positive operators. Then \[\|A^rXB^r\|\leq \|AXB\|^r\|X\|^{1-r},\] where $r\in [0,1].$
\end{lemma}

\begin{proposition}\label{prop-1}

	Let $T\in B(\mathbb{H}).$ Then  (i) $T^2=0$ and (ii) $\widetilde{T}_t=0$, $t\in [0,1]$ are equivalent.

\end{proposition}

\begin{proof}
	We first prove the easier part $ (ii) \Rightarrow (i).$ It follows from the fact that  $T^2=U|T|U|T|=U|T|^{1-t}|T|^tU|T|^{1-t}|T|^t=U|T|^{1-t} \widetilde{T}_t |T|^t$ for any $ t \in [0,1].$\\
	We next prove $ (i) \Rightarrow (ii).$ We claim that 
	 \[ \|\widetilde{T}_t\| \leq \begin{cases}
	\|T^2\|^t\|T\|^{1-2t},& 0\leq t\leq \frac{1}{2}\\
	\|T^2\|^{1-t}\|T\|^{2t-1},& \frac{1}{2}\leq t\leq 1.
	\end{cases} \] 
 Consider $0\leq t\leq \frac{1}{2}$. Then $ \|\widetilde{T}_t\|=\||T|^tU|T|^{1-t}\| $. Using Lemma \ref{Heinz}, we get 
		     $\|\widetilde{T}_t\| \leq \||T|^tU|T|^{t}\| \||T|^{1-2t}\|\leq \||T|U|T|\|^{t}\|U\|^{1-t} \|T\|^{1-2t}=\|T^2\|^t \|T\|^{1-2t}.$\\
		     Next consider  $\frac{1}{2}\leq t\leq 1$. Then  using Lemma \ref{Heinz}, we get  $ \|\widetilde{T}_t\|= \||T|^tU|T|^{1-t}\| \leq \||T|^{2t-1}\| \||T|^{1-t}U|T|^{1-t}\| \leq \|T\|^{2t-1} \||T|U|T|\|^{1-t}\|U\|^{t}= \|T\|^{2t-1} \|T^2\|^{1-t}. $ The proof now easily follows from the claim established.\\

\end{proof}
Next we present the following numerical radius inequality in terms of the Aluthge transform, which improves on one of the upper bound obtained by Yamazaki in \cite[Th. 2.1]{Y}. 

\begin{theorem} \label{theorem:upper1}
	(i) Let $T\in B(\mathbb{H}).$ Then
		\[ w(T)\leq  \min_{t\in [0,1]}\left\{\frac{1}{2}w(\widetilde{T_t})+\frac{1}{4}\left(\|T\|^{2t}+\|T\|^{2-2t}\right)\right\}.\]  In particular, \[ w(T)\leq\frac{1}{2}w(\widetilde{T})+\frac{1}{2}\|T\|.\]

(ii) If dim $\mathbb{H} < \infty$  then the equalities hold in the above inequalities if and only if $T$ is either unitarily similar to $[a] \oplus B,$  $\|B\|\leq |a|$ or to $\left(\begin{array}{cc}
a&0\\
b&0
\end{array}\right) \oplus C,$  $\|C\|\leq (|a|^2+|b|^2)^{\frac{1}{2}}$ and $w(\widetilde{C_t})\leq |a|.$ \\
(iii) When $\mathbb{H}$ is an arbitrary Hilbert space, then the equalities hold  if $T^2=0$ or $T$ is normaloid, i.e., $w(T) = \|T\|.$
\end{theorem}

\begin{proof}
(i) 	It follows from arithmetic-geometric mean inequality that $2\|T\|\leq \|T\|^{2t}+\|T\|^{2-2t}$ for all $t\in [0,1]$.
	Using this and inequality (\ref{number4}), we get 
	\[w(T)\leq\min_{t\in [0,1]}\left\{\frac{1}{2}w(\widetilde{T_t})+\frac{1}{4}\left(\|T\|^{2t}+\|T\|^{2-2t}\right)\right\}.\]
	Considering $t = \frac{1}{2},$ we get 
 \[	w(T) \leq \frac{1}{2}w(\widetilde{T})+\frac{1}{2}\|T\|.\]
 (ii) Let us assume that $T$ is an $n\times n$ matrix. Then following  \cite[Th. 4.2]{GW} we can conclude that the equalities hold if and only if  $T$ is either unitarily similar to $[a] \oplus B,$  $\|B\|\leq |a|$ or to $\left(\begin{array}{cc}
 a&0\\
 b&0
 \end{array}\right) \oplus C,$  $\|C\|\leq (|a|^2+|b|^2)^{\frac{1}{2}}$ and $w(\widetilde{C_t})\leq |a|.$ \\
 (iii) The proof is obvious.\\
\end{proof}

Next we prove the following inequality for the numerical radius which improves on the upper bound  obtained by Kittaneh in \cite[Th. 1]{FK1}.

\begin{theorem}\label{theorem:upper2}
	Let $T\in B(\mathbb{H}).$ Then
	\[w^2(T) \leq\frac{1}{2}\|T\| \left(\min_{t\in [0,1]}\|\widetilde{T_t}\|\right)+\frac{1}{4}\|T^{*}T+TT^{*}\|.\]
	In particular, \[w^2(T)	\leq\frac{1}{2}\|T\|\|\widetilde{T}\|+\frac{1}{4}\|T^{*}T+TT^{*}\|.\]
\end{theorem}

\begin{proof}
	Since $H_{\theta} = \frac{1}{2}(e^{{\rm i} \theta} T + e^{-{\rm i} \theta} T^*)$ for all $\theta \in \mathbb{R}$, we have
	\begin{eqnarray*}
		4 {H_{\theta}}^2 &= & e^{2{\rm i} \theta} T^2 + e^{-2{\rm i} \theta} {T^*}^2 + T^{*}T+TT^{*} \\
		&= & e^{2{\rm i} \theta} U|T|U|T| + e^{-2{\rm i} \theta} |T|U^{*}|T|U^{*} + T^{*}T+TT^{*} \\
		&= & e^{2{\rm i} \theta} U|T|^{1-t}|T|^{t}U|T|^{1-t}|T|^{t} + e^{-2{\rm i} \theta} |T|^{t}|T|^{1-t}U^{*}|T|^{t}|T|^{1-t}U^{*} \\
		&& +(T^{*}T+TT^{*}) \\
		&= & e^{2{\rm i} \theta} U|T|^{1-t}\widetilde{T_t}|T|^{t} + e^{-2{\rm i} \theta} |T|^{t}\widetilde{T^{*}_t}|T|^{1-t}U^{*} + T^{*}T+TT^{*}.
	\end{eqnarray*}
	Hence, 
	\begin{eqnarray*}
		4 \|{H_{\theta}}\|^2  &\leq & \|e^{2{\rm i} \theta} U|T|^{1-t}\widetilde{T_t}|T|^{t}\| + \|e^{-2{\rm i} \theta} |T|^{t}\widetilde{T^{*}_t}|T|^{1-t}U^{*}\| + \|T^{*}T+TT^{*}\| \\
		&\leq & 2\|T\|	\|\widetilde{T_t}\|+ \|T^{*}T+TT^{*}\|. 
	\end{eqnarray*}
	Therefore,
	\begin{eqnarray*}
		\|{H_{\theta}}\|^2  &\leq & \frac{1}{2}\|T\|	\|\widetilde{T_t}\|+\frac{1}{4} \|T^{*}T+TT^{*}\|.
	\end{eqnarray*}
	Taking supremum over $\theta \in \mathbb{R}$ in the above inequality and then using Lemma \ref{lemma:htheta}, we get
	\[w^2(T)  \leq  \frac{1}{2}\|T\|	\|\widetilde{T_t}\|+\frac{1}{4} \|T^{*}T+TT^{*}\|.\]
	This inequality holds for all $t\in [0,1]$, and so taking minimum we get,
	\[w^2(T)\leq \frac{1}{2}\|T\|\big(\min_{t\in [0,1]}\|\widetilde{T_t}\|\big)+\frac{1}{4}\|T^{*}T+TT^{*}\|.\]
	Considering the case  $t=\frac{1}{2},$ we get 
	\[w^2(T)	\leq\frac{1}{2}\|T\|\|\widetilde{T}\|+\frac{1}{4}\|T^{*}T+TT^{*}\|.\]
\end{proof}

\begin{remark}
	If $T^2=0$  or $T$ is a normaloid operator then   inequalities  in Theorem \ref{theorem:upper2} become equalities. If $T^2=0$ then $w(T)=\frac{1}{2}\sqrt{\|T^*T+TT^*\|}$, (see \cite[Th. 2.3]{BBP}) and $ \frac{1}{2}\|T\|\big(\min_{t\in [0,1]}\|\widetilde{T_t}\|\big)+\frac{1}{4}\|T^{*}T+TT^{*}\| =  \frac{1}{4}\|T^{*}T+TT^{*}\|.$ Thus we get the equalities if $T^2=0.$ 
	Also note that  $ w^2(T) \leq \frac{1}{2}\|T\|\big(\min_{t\in [0,1]}\|\widetilde{T_t}\|\big)+\frac{1}{4}\|T^{*}T+TT^{*}\| \leq \|T\|^2$ and so normaloid condition forces the inequlities to be equalities.
\end{remark}

\begin{remark}\label{remark1}
	Kittaneh in \cite[Th. 1]{FK1} proved that for $T\in B(\mathbb{H})$,
	$$ w(T)\leq \frac{1}{2}\left(\|T\|+\|T^2\|^{\frac{1}{2}}\right).$$
	We know that  $\|\widetilde{T}\|\leq \|T^2\|^{\frac{1}{2}} $ (see \cite[(2.1)]{Y}) and  $\|T^{*}T+TT^{*}\|\leq \|T\|^2+\|T^2\|$ (see \cite[Lemma 7]{FK3}) and so from our Theorem \ref{theorem:upper2}, we get
	$$w^2(T)\leq \frac{1}{2}\|T\|\|T^2\|^{\frac{1}{2}}+\frac{1}{4}\left(\|T\|^2+\|T^2\|\right).$$ 
	Hence,
	$$w(T)\leq \frac{1}{2}\left(\|T\|+\|T^2\|^{\frac{1}{2}}\right).$$
	Thus  our bound obtained in Theorem \ref{theorem:upper2} is better than bound (\ref{number1}) obtained by Kittaneh in \cite[Th. 1]{FK1}. Also  there are operators for which bound obtained by us in  Theorem \ref{theorem:upper2} is better than that obtained in inequality (\ref{number4}) obtained by Abu-Omar and Kittaneh  \cite[Th. 3.2]{AF2}. As for example we consider  $T=\left(\begin{array}{ccc}
	0&2&0 \\
	0&0&0\\
	0&0&1
	\end{array}\right).$ It is easy to see that $\|T\|=2$ and $T$ has the polar decomposition $T=U|T|$, where $|T|=\left(\begin{array}{ccc}
	0&0&0 \\
	0&2&0\\
	0&0&1
	\end{array}\right)$  and $U=\left(\begin{array}{ccc}
	0&1&0 \\
	1&0&0\\
	0&0&1
	\end{array}\right).$ Hence $ \widetilde{T_t}=|T|^tU|T|^{1-t}=\left(\begin{array}{ccc}
	0&0&0 \\
	0&0&0\\
	0&0&1
	\end{array}\right)$ for all $t\in [0,1]$, and so $w(\widetilde{T_t})=\|\widetilde{T_t}\|=1$ for all $t\in [0,1]$. It follows that  $$\frac{1}{2}\big(\|T\|+\min_{t\in [0,1]}w(\widetilde{T_t})\big)=\frac{1}{2}(2+1)=\frac{3}{2},$$
	$$\frac{1}{2}\|T\|\big(\min_{t\in [0,1]}\|\widetilde{T_t}\|\big)+\frac{1}{4}\|T^{*}T+TT^{*}\|=\frac{1}{2} \times 2\times 1+\frac{1}{4}\times 4=2.$$
	Therefore, Theorem \ref{theorem:upper2} gives $w(T)\leq \sqrt{2}$, whereas (\ref{number4}) gives $w(T)\leq \frac{3}{2}.$\\
\end{remark}

We next obtain an upper bound for the numerical radius which improves on the bound (\ref{number1}) obtained by Kittaneh in \cite[Th. 1]{FK1}.
To achieve it, we need the following inequality obtained by Abu-Omar and Kittaneh \cite{AF2}.
\begin{theorem}\cite[Th. 2.2]{AF2} \label{theorem:spectral2}
	Let $A_1, A_2, B_1, B_2 \in B(\mathbb{H}).$ Then 
	\begin{eqnarray*}
		r(A_1B_1+A_2B_2) &\leq& \frac{1}{2}\left(w(B_1A_1)+w(B_2A_2)\right)\\
		&& +\frac{1}{2}\sqrt{\left(w(B_1A_1)-w(B_2A_2)\right)^2+4\|B_1A_2\| \|B_2A_1\|}.
	\end{eqnarray*}
\end{theorem}

We now prove the following theorem.
\begin{theorem}\label{theorem:upper3}
	Let $T\in B(\mathbb{H}).$ Then
	\[ w^2(T)\leq \min_{t\in [0,1]}\left(\frac{1}{4}w(\widetilde{T_t}^2)+\frac{1}{4}\|T\|\|\widetilde{T_t}\|\right)+\frac{1}{4}\|T^{*}T+TT^{*}\|\]
		In particular, \[ w^2(T)\leq \frac{1}{4}w(\widetilde{T}^2)+\frac{1}{4}\|T\|\|\widetilde{T}\|+\frac{1}{4}\|T^{*}T+TT^{*}\|.\]
\end{theorem}

\begin{proof}
	Since $H_{\theta} =  \frac{1}{2}(e^{{\rm i} \theta} T + e^{-{\rm i} \theta} T^*)$ for all $\theta \in \mathbb{R}$, we have
	\begin{eqnarray*}
		4 {H_{\theta}}^2  &= & e^{2{\rm i} \theta} T^2 + e^{-2{\rm i} \theta} {T^*}^2 + T^{*}T+TT^{*} \\
		&= & e^{2{\rm i} \theta} U|T|U|T| + e^{-2{\rm i} \theta} |T|U^{*}|T|U^{*} + T^{*}T+TT^{*} \\
		&= & e^{2{\rm i} \theta} U|T|^{1-t}|T|^{t}U|T|^{1-t}|T|^{t} + e^{-2{\rm i} \theta} |T|^{t}|T|^{1-t}U^{*}|T|^{t}|T|^{1-t}U^{*}\\
		&& +T^{*}T+TT^{*}. 
	\end{eqnarray*}
	Hence, 
	\begin{eqnarray*}
		4 \|{H_{\theta}}\|^2  &\leq & \|e^{2{\rm i} \theta} U|T|^{1-t}|T|^{t}U|T|^{1-t}|T|^{t} + e^{-2{\rm i} \theta} |T|^{t}|T|^{1-t}U^{*}|T|^{t}|T|^{1-t}U^{*}\| \\
		&& +\|T^{*}T+TT^{*}\| \\
		&= & r\big(e^{2{\rm i} \theta} U|T|^{1-t}|T|^{t}U|T|^{1-t}|T|^{t} + e^{-2{\rm i} \theta} |T|^{t}|T|^{1-t}U^{*}|T|^{t}|T|^{1-t}U^{*}\big) \\
		&&+ \|T^{*}T+TT^{*}\|,~~ r(S)=\|S\| ~~\textit{for hermitian operator} ~~S  \\
		& = & r\big(A_1B_1+A_2B_2\big)+\|T^{*}T+TT^{*}\| ,
	\end{eqnarray*}			
	where   $ A_1 = e^{2{\rm i} \theta} U|T|^{1-t}|T|^{t}U|T|^{1-t}, B_1=|T|^{t}, 
	A_2 = e^{-2{\rm i} \theta} |T|^{t} $ and\\ $ B_2=|T|^{1-t}U^*|T|^{t}|T|^{1-t}U^*. $
	Then using Theorem \ref{theorem:spectral2}, we get	
	\begin{eqnarray*}
		4 \|{H_{\theta}}\|^2 &\leq & w(\widetilde{T_t}^2)+\sqrt{\||T|^{2t}\| \|\widetilde{T_t}^{*}|T|^{1-t}U^{*}U|T|^{1-t}\widetilde{T_t} \|}+\|T^{*}T+TT^{*}\| \\
		&= & w(\widetilde{T_t}^2)+\sqrt{\|T\|^{2t} \|\widetilde{T_t}^{*}|T|^{2-2t} \widetilde{T_t} \|}+\|T^{*}T+TT^{*}\| \\
		&\leq & w(\widetilde{T_t}^2)+\sqrt{\|T\|^{2t} \|\widetilde{T_t}\|^2 \|T\|^{2-2t} }+\|T^{*}T+TT^{*}\| \\
		&=& w(\widetilde{T_t}^2)+ \|T\| \|\widetilde{T_t}\| +\|T^{*}T+TT^{*}\|. 
	\end{eqnarray*}			
	Taking supremum over $\theta \in \mathbb{R}$ in the above inequality and then using Lemma \ref{lemma:htheta}, we get
	\[w^2(T)\leq \frac{1}{4}w(\widetilde{T_t}^2)+\frac{1}{4}\|T\|\|\widetilde{T_t}\|+\frac{1}{4}\|T^{*}T+TT^{*}\|.\]
	This  holds for all $t\in [0,1]$, and so taking minimum we get,
	\[w^2(T)\leq \min_{t\in [0,1]}\left(\frac{1}{4}w(\widetilde{T_t}^2)+\frac{1}{4}\|T\|\|\widetilde{T_t}\|\right)+\frac{1}{4}\|T^{*}T+TT^{*}\|.\]
	Considering the case  $t=\frac{1}{2},$ we get 
	\[w^2(T)\leq \frac{1}{4}w(\widetilde{T}^2)+\frac{1}{4}\|T\|\|\widetilde{T}\|+\frac{1}{4}\|T^{*}T+TT^{*}\|.\]	
\end{proof}

\begin{remark}
	  We observe that  $ \min_{t\in [0,1]}\left(\frac{1}{4}w(\widetilde{T_t}^2)+\frac{1}{4}\|T\|\|\widetilde{T_t}\|\right) = 0 $ if $T^2=0.$ Also, as discussed in Remark 1, if   $T^2=0$  or $T$ is a normaloid operator then   inequalities  in Theorem \ref{theorem:upper3} become equalities. 
\end{remark}

\begin{remark}
	It is easy to observe that the inequality obtained by us in Theorem \ref{theorem:upper3} is sharper than the inequality obtained  in Theorem \ref{theorem:upper2} and so it is sharper than inequality (\ref{number1}) obtained by Kittaneh in \cite[Th. 1]{FK1}. Also if we take the same matrix $T=\left(\begin{array}{ccc}
	0&2&0 \\
	0&0&0\\
	0&0&1
	\end{array}\right)$   as in Remark \ref{remark1} then Theorem \ref{theorem:upper3} gives $w(T)\leq \sqrt{\frac{7}{4}}$, whereas (\ref{number4}) gives $w(T)\leq \frac{3}{2}.$ Thus for this matrix, our inequality obtained in Theorem \ref{theorem:upper3} is better than inequality (\ref{number4}) obtained by Abu-Omar and Kittaneh  \cite[Th. 3.2]{AF2}. Infact, if we consider $T=\left(\begin{array}{ccc}
	0&a&0 \\
	0&0&0\\
	0&0&b
	\end{array}\right)$ where $a,b \in \mathbb{C}$, then we see that the bound in Theorem \ref{theorem:upper3} is always less than or equal to the bound in \cite[Th. 3.2]{AF2}, obtained by Abu-Omar and Kittaneh.
\end{remark}

Next using  Theorem \ref{theorem:upper3} we obtain the following inequality for the numerical radius in terms of  iterated $t$-Aluthge transform. For a non-negative integer $n$, we denote the $n$th iterated $t$-Aluthge transform ${\widetilde{T}_{t_n}}$, i.e., ${\widetilde{T}_{t_n}}={\widetilde{\widetilde{T}}_{t_{n-1}}}$ and ${\widetilde{T}_{t_0}}=T$.

\begin{theorem}\label{iterate-1}
	Let $T\in B(\mathbb{H})$. Then
	\[w^2(T)\leq \sum_{n=1}^{\infty}\frac{1}{4^n}\left( \|\widetilde{T}_{t_{n-1}}\| \|\widetilde{T}_{t_{n}}\|+\|\widetilde{T}_{t_{n-1}}^*\widetilde{T}_{t_{n-1}}+\widetilde{T}_{t_{n-1}}\widetilde{T}_{t_{n-1}}^*\|    \right),\] for all $t\in [0,1].$
\end{theorem}

\begin{proof}
	By using Theorem \ref{theorem:upper3} repeatedly, we get
	\begin{eqnarray*}
		w^2(T)&\leq& \frac{1}{4}\left(\|T\| \|\widetilde{T}_t\|+\|T^*T+TT^*\| \right)+\frac{1}{4}w(\widetilde{T}_t^2)\\
		&\leq& \frac{1}{4}\left(\|T\| \|\widetilde{T}_t\|+\|T^*T+TT^*\| \right)+\frac{1}{4}w^2(\widetilde{T}_t)\\
		&\leq& \frac{1}{4}\left(\|T\| \|\widetilde{T}_t\|+\|T^*T+TT^*\| \right)\\
		&&+\frac{1}{4^2}\left(\|\widetilde{T}_t\| \|\widetilde{T}_{t_2}\|+\|\widetilde{T}_t^*\widetilde{T}_t+\widetilde{T}_t\widetilde{T}_t^*   \|   \right)+\frac{1}{4^2}w(\widetilde{T}_2^2)\\
		&\leq& \frac{1}{4}\left(\|T\| \|\widetilde{T}_t\|+\|T^*T+TT^*\| \right)\\
		&&+\frac{1}{4^2}\left(\|\widetilde{T}_t\| \|\widetilde{T}_{t_2}\|+\|\widetilde{T}_t^*\widetilde{T}_t+\widetilde{T}_t\widetilde{T}_t^*   \|   \right)+\frac{1}{4^2}w^2(\widetilde{T}_2)\\
		&\leq& \frac{1}{4}\left(\|T\| \|\widetilde{T}_t\|+\|T^*T+TT^*\| \right)\\
		&&+\frac{1}{4^2}\left(\|\widetilde{T}_t\| \|\widetilde{T}_{t_2}\|+\|\widetilde{T}_t^*\widetilde{T}_t+\widetilde{T}_t\widetilde{T}_t^*   \|   \right)\\
		&& +\frac{1}{4^3}\left(\|\widetilde{T}_{t_2}\| \|\widetilde{T}_{t_3}\|+\|\widetilde{T}_{t_2}^*\widetilde{T}_{t_2}+\widetilde{T}_{t_2}\widetilde{T}_{t_2}^*   \|   \right)+\frac{1}{4^3}w(\widetilde{T}_3^2)\\
		&\leq & \ldots\\
		&\leq& \sum_{n=1}^{\infty}\frac{1}{4^n}\left( \|\widetilde{T}_{t_{n-1}}\| \|\widetilde{T}_{t_{n}}\|+\|\widetilde{T}_{t_{n-1}}^*\widetilde{T}_{t_{n-1}}+\widetilde{T}_{t_{n-1}}\widetilde{T}_{t_{n-1}}^*\|    \right).
	\end{eqnarray*}
\end{proof}

Using Theorem \ref{iterate-1}, we obtain the following inequality.

\begin{corollary}\label{cor-2}
	Let $T\in B(\mathbb{H}).$ Then 
	\[w^2(T) \leq \frac{1}{2}\left[\|T^2\|^{\frac{1}{2}}\left(\frac{1}{2}\|T\|+\frac{1}{2}\|T^2\|^{\frac{1}{2}}\right)+\frac{1}{2}\|T^*T+TT^*\|  \right]. \]
\end{corollary}
\begin{proof}
	Let $\widetilde{T}_{n}$ be the $n$-th iterated Aluthge transform. Then from Theorem \ref{iterate-1} $( \textit {for} ~~t=\frac{1}{2})$, we get
	\begin{eqnarray*}
	w^2(T)&\leq& \sum_{n=1}^{\infty}\frac{1}{4^n}\left( \|\widetilde{T}_{{n-1}}\| \|\widetilde{T}_{{n}}\|+\|\widetilde{T}_{{n-1}}^*\widetilde{T}_{{n-1}}+\widetilde{T}_{{n-1}}\widetilde{T}_{{n-1}}^*\|    \right)\\
	&=& \frac{1}{4}\left(\|T\| \|\widetilde{T}\|+\|T^*T+TT^*\|\right)+\sum_{n=2}^{\infty}\frac{1}{4^n}\left( \|\widetilde{T}_{{n-1}}\| \|\widetilde{T}_{{n}}\|+\|\widetilde{T}_{{n-1}}^*\widetilde{T}_{{n-1}}+\widetilde{T}_{{n-1}}\widetilde{T}_{{n-1}}^*\|    \right)\\
	&\leq& \frac{1}{4}\left(\|T\| \|\widetilde{T}\|+\|T^*T+TT^*\|\right)+\sum_{n=2}^{\infty}\frac{1}{4^n}\left( \|\widetilde{T}_{{n-1}}\| \|\widetilde{T}_{{n}}\|+2\|\widetilde{T}_{{n-1}}\|^2    \right)\\
	&\leq& \frac{1}{4}\left(\|T\| \|\widetilde{T}\|+\|T^*T+TT^*\|\right)+\sum_{n=2}^{\infty}\frac{1}{4^n}\left( 3\|\widetilde{T}\|^2    \right), ~~\textit{using} ~~\|\widetilde{T}_{n}\|\leq \|\widetilde{T}_{n-1}\|,~~ n\geq 2\\
	&\leq& \frac{1}{4}\left(\|T\| \|{T}^2\|^{\frac{1}{2}}+\|T^*T+TT^*\|\right)+\sum_{n=2}^{\infty}\frac{1}{4^n}\left( 3\|{T}^2\|    \right), ~~\textit{using} ~\|\widetilde{T}\|\leq \|T^2\|^{\frac{1}{2}}\\
	&=& \frac{1}{4}\left(\|T\| \|{T}^2\|^{\frac{1}{2}}+\|T^*T+TT^*\|\right)+ \frac{3}{4^2}\|{T}^2\| \sum_{n=0}^{\infty}\frac{1}{4^n}\\
	&=& \frac{1}{4}\left(\|T\| \|{T}^2\|^{\frac{1}{2}}+\|T^*T+TT^*\|\right)+ \frac{1}{4}\|{T}^2\|\\
	&=& \frac{1}{2}\left[\|T^2\|^{\frac{1}{2}}\left(\frac{1}{2}\|T\|+\frac{1}{2}\|T^2\|^{\frac{1}{2}}\right)+\frac{1}{2}\|T^*T+TT^*\|  \right].
	\end{eqnarray*}
\end{proof}

\begin{remark}   Our inequality in Corollary \ref{cor-2} is better than inequality (\ref{number1}), it follows from the fact that 
	\begin{eqnarray*}
		&& \frac{1}{2}\left[\|T^2\|^{\frac{1}{2}}\left(\frac{1}{2}\|T\|+\frac{1}{2}\|T^2\|^{\frac{1}{2}}\right)+\frac{1}{2}\|T^*T+TT^*\|  \right]\\
		&=& \frac{1}{4}\|T^2\|^{\frac{1}{2}}\|T\|+\frac{1}{4} \|T^2\|+ \frac{1}{4}\|T^*T+TT^*\|\\
		&\leq& \frac{1}{4}\|T^2\|^{\frac{1}{2}}\|T\|+\frac{1}{4} \|T^2\|+ \frac{1}{4}\|T^2\|+\frac{1}{4}\|T\|^2\\
		&\leq& \frac{1}{2}\|T^2\|^{\frac{1}{2}}\|T\|+\frac{1}{4} \|T^2\|+ \frac{1}{4}\|T\|^2\\
		&=& \left(\frac{1}{2}\|T\|+\frac{1}{2}\|T^2\|^{\frac{1}{2}}\right)^2.
	\end{eqnarray*}
	We also observe that  bound obtained  in Corollary \ref{cor-2} is sharper than that in the right hand inequality  of (\ref{number2}), if 
	$ \|T\|\|T^2\|^{\frac{1}{2}}+\|T^2\|\leq \|TT^*+T^*T\|.$
\end{remark}

Next we obtain an upper bound for the  numerical radius and give an example to show that this bound  improves on bound (\ref{number3}).

\begin{theorem}\label{theorem:upper4}
	Let $T\in B(\mathbb{H}).$ Then
	\[ w^4(T)\leq \frac{1}{16} \min_{t\in [0,1]}\left(w(\widetilde{T_t}^2)+\|T\|\|\widetilde{T_t}\|\right)^2+\frac{1}{8}w(T^2P+PT^2)+\frac{1}{16}\|P\|^2,\]
	where $P=T^{*}T+TT^{*}.$ In particular, \[ w^4(T)	\leq \frac{1}{16}\left(w({\widetilde{T}}^2)+\|T\|\|\widetilde{T}\|\right)^2+\frac{1}{8}w(T^2P+PT^2)+\frac{1}{16}\|P\|^2.\]
	\end{theorem}

\begin{proof}
	Since $H_{\theta} = \frac{1}{2}(e^{{\rm i} \theta} T + e^{-{\rm i} \theta} T^*)$ for all $\theta \in \mathbb{R}$, we have
	\begin{eqnarray*}
		4 {H_{\theta}}^2  &= & e^{2{\rm i} \theta} T^2 + e^{-2{\rm i} \theta} {T^*}^2 + P \\
		\Rightarrow 16 {H_{\theta}}^4  &= & \big(e^{2{\rm i} \theta} T^2 + e^{-2{\rm i} \theta} {T^*}^2\big)^2 +2\textit{Re}(e^{2{\rm i} \theta}(T^2P+PT^2))+ P^2.
	\end{eqnarray*}
	Hence, 
	\begin{eqnarray*}
		16 \|{H_{\theta}}\|^4  &\leq & \|e^{2{\rm i} \theta} T^2 + e^{-2{\rm i} \theta} {T^*}^2\|^2 +2\|\textit{Re}(e^{2{\rm i} \theta}(T^2P+PT^2))\|+ \|P\|^2 \\
		&\leq & r^2\big(e^{2{\rm i} \theta} T^2 + e^{-2{\rm i} \theta} {T^*}^2\big) +2 w(T^2P+PT^2)+ \|P\|^2,\\
		&& \,\,\,\,\,  r(S)=\|S\|~~ \textit{for hermitian operator}~~ S \\
		&=& r^2\big(e^{2{\rm i} \theta} U|T|U|T| + e^{-2{\rm i} \theta} |T|U^*|T|U^* \big) +2 w(T^2P+PT^2)+ \|P\|^2.
	\end{eqnarray*}
	Then using the same technique as  in Theorem \ref{theorem:upper3}, we get
	\[ \|{H_{\theta}}\|^4 \leq \frac{1}{16} \big(w(\widetilde{T_t}^2)+\|T\|\|\widetilde{T_t}\|\big)^2+\frac{1}{8}w(T^2P+PT^2)+\frac{1}{16}\|P\|^2.\]
	Taking supremum over $\theta \in \mathbb{R}$ in the above inequality and then using Lemma \ref{lemma:htheta}, we get
	\[w^4(T)\leq \frac{1}{16}\big(w(\widetilde{T_t}^2)+\|T\|\|\widetilde{T_t}\|\big)^2+\frac{1}{8}w(T^2P+PT^2)+\frac{1}{16}\|P\|^2.\]
	This holds for all $t\in [0,1]$,  and so taking minimum we get,
	\[w^4(T)\leq \frac{1}{16} \min_{t\in [0,1]}\big(w(\widetilde{T_t}^2)+\|T\|\|\widetilde{T_t}\|\big)^2+\frac{1}{8}w(T^2P+PT^2)+\frac{1}{16}\|P\|^2.\]
	Considering the case  $t=\frac{1}{2},$ we get 
	\[w^4(T)\leq \frac{1}{16}\big(w({\widetilde{T}}^2)+\|T\|\|\widetilde{T}\|\big)^2+\frac{1}{8}w(T^2P+PT^2)+\frac{1}{16}\|P\|^2.\]
\end{proof}

\begin{remark}
	We observe that as discussed in Remark 1, if $T^2=0$  or $T$ is a normaloid operator then   inequalities  in Theorem \ref{theorem:upper4} become equalities.
\end{remark}

We now give an example to show that the bound obtained  in Theorem \ref{theorem:upper4} improves on bound (\ref{number3})  obtained by Yamazaki in \cite [Th. 2.1]{Y}.

\begin {example}
We consider  $T=\left(\begin{array}{ccc}
0&2&0 \\
0&0&3\\
0&0&0
\end{array}\right).$
Then it is easy to see that $P=\left(\begin{array}{ccc}
4&0&0 \\
0&13&0\\
0&0&9
\end{array}\right), $ $|T|=\left(\begin{array}{ccc}
0&0&0 \\
0&2&0\\
0&0&3
\end{array}\right)$ and $U=\left(\begin{array}{ccc}
0&1&0 \\
0&0&1\\
0&0&0
\end{array}\right)$,
where $U$ is the partial isometry in the polar decomposition of $T$, i.e., $ T=U|T|$. So, 
\[\widetilde{T_t}=  |T|^tU|T|^{1-t}=\left(\begin{array}{ccc}
0&0&0 \\
0&0&2^t3^{1-t}\\
0&0&0
\end{array}\right).\]
Therefore, $w(\widetilde{T_t})=\frac{2^t3^{1-t}}{2}$, $\|\widetilde{T_t}\|=2^t3^{1-t} $, $\|P\|=13$ and $w(T^2P+PT^2)=39.$ So, the inequality obtained by us in Theorem \ref{theorem:upper4} gives $w(T)\leq 2.05076838.$ But inequality (\ref{number3})  obtained by Yamazaki in \cite [Th. 2.1]{Y} gives $w(T)\leq 2.11237244.$ 
\end{example}

Our next goal is to improve on both upper and lower bound of the numerical radius obtained by Kittaneh in \cite[Th. 1]{FK2}. Before doing so, we first give an alternative proof of the following theorem proved by Kittaneh in \cite [Th. 1] {FK2}.

\begin{theorem} \cite [Th. 1]{FK2} \label{theorem:upper6}
Let $T\in B(\mathbb{H})$, then 
\[\frac{1}{4}\|T^* T+TT^*\| \leq w^2(T)\leq \frac{1}{2}\|T^* T+TT^*\|.\]
\end{theorem}

\begin{proof} Since $H_{\theta} = \frac{1}{2}(e^{{\rm i} \theta} T + e^{-{\rm i} \theta} T^*)$ and $K_\theta=\frac{1}{2 {\rm i}}(e^{{\rm i} \theta} T - e^{-{\rm i} \theta} T^*)$ for all $\theta \in \mathbb{R}$, we have  
$ H^2_\theta+K^2_\theta =\frac{1}{2}(T^{*}T+TT^{*}) $ and so $ \frac{1}{2}\|T^{*}T+TT^{*}\| = \|H^2_\theta+K^2_\theta\|
\leq \|H_\theta\|^2+\|K_\theta\|^2 \leq 2w^2(T)$, using Lemma \ref{lemma:htheta}. Thus $\frac{1}{4}\|T^{*}T+TT^{*}\| \leq w^2(T).$							
This completes the proof of the first inequality. \\
Again, from  $H^2_\theta+K^2_\theta = \frac{1}{2}(T^{*}T+TT^{*})$ we get, $H^2_\theta - \frac{1}{2}(T^{*}T+TT^{*})=-K^2_\theta \leq 0.$
Thus $ H_{\theta} ^2 \leq  \frac{1}{2}(T^* T+TT^*) $  and so $ \|H_{\theta} ^2\| \leq  \frac{1}{2}\|T^* T+TT^*\|. $ Taking supremum over $\theta \in \mathbb{R}$  and then using Lemma \ref{lemma:htheta}, we get
$ w^2(T)\leq \frac{1}{2}\|T^* T+TT^*\|.$

\end{proof}

We now prove the desired inequalitiy which  improves on inequality (\ref{number2}) obtained by Kittaneh in \cite[Th. 1]{FK2}.

\begin{theorem} \label {theorem:upper7}
Let $T \in B(\mathbb{H}).$ Then 
\[\frac{1}{4}m\big(\left(\mbox{Re}(T^2)\right)^2\big)+\frac{1}{16}\|T^{*}T+TT^{*}\|^2  \leq w^4(T) \leq \frac{1}{2}w^2(T^2)+\frac{1}{8}\|T^{*}T+TT^{*}\|^2.\]
\end{theorem}

\begin{proof} We first prove the left hand inequality. 
Let $x \in {\mathbb{H}}$ with $\|x\|=1$. Since $H_{\theta} = \frac{1}{2}(e^{{\rm i}\theta} T + e^{-{\rm i}\theta} T^*)$ and $K_\theta=\frac{1}{2{\rm i}}(e^{{\rm i}\theta} T - e^{-{\rm i}\theta} T^*)$ for all $\theta \in \mathbb{R}$, we have 
\begin{eqnarray*}
	\frac{1}{8}\left[4 \left(\textit{Re}(e^{2{\rm i} \theta} T^2)\right)^2+ \left(T^{*}T+TT^{*}\right)^2\right] & = & H^4_\theta+K^4_\theta \\
	\Rightarrow \frac{1}{2} \langle \left(\textit{Re}(e^{2{\rm i} \theta} T^2)\right)^2 x,x\rangle+ \frac{1}{8} \langle \left(T^{*}T+TT^{*}\right)^2x,x\rangle &=& \langle H^4_\theta x,x\rangle+ \langle K^4_\theta x,x\rangle \\
	\Rightarrow \frac{1}{2} \langle \left(\textit{Re}(e^{2{\rm i} \theta} T^2)\right)^2x,x\rangle + \frac{1}{8} \langle \left(T^{*}T+TT^{*}\right)^2x,x\rangle &\leq& 2w^4(T).
\end{eqnarray*}
This inequality holds for all $\theta \in \mathbb{R}$. So taking $\theta =0$, we get
\begin{eqnarray*}
	\frac{1}{2} \langle \left(\textit{Re}(T^2)\right)^2x,x\rangle+ \frac{1}{8}\langle \left(T^{*}T+TT^{*}\right)^2x,x\rangle &\leq& 2w^4(T)\\
	\Rightarrow \frac{1}{2} m\big((\textit{Re}(T^2))^2\big)+ \frac{1}{8} \langle \left(T^{*}T+TT^{*}\right)^2x,x\rangle &\leq& 2w^4(T).
\end{eqnarray*}
Taking supremum over $x\in \mathbb{H}, \|x\|=1$, we get
\begin{eqnarray*}
	\frac{1}{2} m\big(\left(\textit{Re}(T^2)\right)^2\big)+ \frac{1}{8}\|T^{*}T+TT^{*}\|^2 &\leq& 2w^4(T).
\end{eqnarray*}
Thus,
\begin{eqnarray*}
	\frac{1}{4} m\big(\left(\textit{Re}(T^2)\right)^2\big)+ \frac{1}{16}\|T^{*}T+TT^{*}\|^2 &\leq& w^4(T).
\end{eqnarray*}
This completes the proof of the left hand inequality. \\
We next prove the right hand inequality.  As before, we have 
\begin{eqnarray*}
	H^4_\theta+K^4_\theta &=& \frac{1}{8}\left[4 \left(\textit{Re}(e^{2{\rm i} \theta} T^2)\right)^2+ \left(T^{*}T+TT^{*}\right)^2\right]
\end{eqnarray*}
and so 
\[ \frac{1}{8}\left[4 \left(\textit{Re}(e^{2{\rm i} \theta} T^2)\right)^2+ \left(T^{*}T+TT^{*}\right)^2\right]-H^4_\theta=K^4_\theta\geq 0.\]
Hence, 
\begin{eqnarray*}
	H^4_\theta & \leq & \frac{1}{8}\left[4 \left(\textit{Re}(e^{2{\rm i} \theta} T^2)\right)^2+ \left(T^{*}T+TT^{*}\right)^2\right].\\
\end{eqnarray*}
Therefore,
\begin{eqnarray*}
	\|H_\theta\|^4 &\leq& \frac{1}{8}\left \|\left[4 \left(\textit{Re}(e^{2{\rm i} \theta} T^2)\right)^2+ \left(T^{*}T+TT^{*}\right)^2\right]\right\|\\
	&\leq& \frac{1}{8}\left[4 \|\textit{Re}(e^{2{\rm i} \theta} T^2)\|^2+ \|T^{*}T+TT^{*}\|^2\right]\\
	&\leq& \frac{1}{8}\left[4 w^2(T^2)+ \|T^{*}T+TT^{*}\|^2\right] ,~~  \mbox{using Lemma}~~ \ref{lemma:htheta}.
\end{eqnarray*}
Taking supremum over $\theta \in \mathbb{R}$ in the above inequality and then using Lemma \ref{lemma:htheta}, we get
\[ w^4(T)\leq \frac{1}{2}w^2(T^2)+\frac{1}{8}\|T^{*}T+TT^{*}\|^2.\]

\end{proof}

\begin{remark}
Clearly the left hand inequality obtained in Theorem \ref{theorem:upper7} is sharper than that of (\ref{number2}) obtained by Kittaneh in \cite[Th. 1]{FK2}. To claim the same for the right hand inequality we first note that $2\|T^2\|\leq \|T^{*}T+TT^{*}\|$ (see \cite{FK3}).   From the right hand inequality obtained in Theorem \ref{theorem:upper7} we get,
\begin{eqnarray*} 
	w^4(T)&\leq &\frac{1}{2}w^2(T^2)+\frac{1}{8}\|T^{*}T+TT^{*}\|^2\\
	&\leq &\frac{1}{2}\|T^2\|^2+\frac{1}{8}\|T^{*}T+TT^{*}\|^2\\
	& = & \frac{1}{8} (2\|T^2\|)^2+\frac{1}{8}\|T^{*}T+TT^{*}\|^2\\
	&\leq & \frac{1}{8} \|T^{*}T+TT^{*}\|^2+\frac{1}{8}\|T^{*}T+TT^{*}\|^2\\
	& = & \frac{1}{4} \|T^{*}T+TT^{*}\|^2.
\end{eqnarray*}
Thus our inequality is sharper than that of Kittaneh \cite[Th. 1]{FK2}.
\end{remark}
We now turn our attention to the bounds that are not comparable in general. The following numerical examples will illustrate the incomparability of some of the upper bounds of the numerical radius.
\begin{example}
(i) \textbf{Incomparabilty of $\frac{1}{2}\left(\|T\|+\|T^2\|^{\frac{1}{2}}\right)$ and $ \sqrt{\frac{1}{2} \|T^{*}T+TT^{*}\|}$}.\\
Consider $T=\left(\begin{array}{cc}
1&1 \\
0&-1
\end{array}\right)$ then $\frac{1}{2}\left(\|T\|+\|T^2\|^{\frac{1}{2}}\right) = \frac{3+\sqrt{5}}{4}$, whereas  $ \sqrt{\frac{1}{2} \|T^{*}T+TT^{*}\|} = \sqrt{\frac{3}{2}}$. Again if we consider  $T=\left(\begin{array}{cc}
1&2 \\
0&-1
\end{array}\right)$ then $\frac{1}{2}\left(\|T\|+\|T^2\|^{\frac{1}{2}}\right) = \frac{2+\sqrt{2}}{2}$, whereas  $\sqrt{\frac{1}{2} \|T^{*}T+TT^{*}\|} = \sqrt{3}$.
This shows that upper bounds obtained in (\ref{number1})  and (\ref{number2})  are not comparable.\\
(ii) \textbf{Incomparabilty of  $ \frac{1}{2}\left(\|T\|+\min_{t\in [0,1]}w(\widetilde{T_t})\right) $ and  $\left(\frac{1}{2}w^2(T^2)+\frac{1}{8}\|T^{*}T+TT^{*}\|^2\right)^{\frac{1}{4}}$.} \\
Consider $T=\left(\begin{array}{ccc}
0&2&0 \\
0&0&0\\
0&0&1
\end{array}\right)$ then $  \left(\frac{1}{2}w^2(T^2)+\frac{1}{8}\|T^{*}T+TT^{*}\|^2\right)^{\frac{1}{4}} =  \sqrt{\sqrt{\frac{5}{2}}}$, whereas \\  $ \frac{1}{2}\left(\|T\|+\min_{t\in [0,1]}w(\widetilde{T_t})\right) = \frac{3}{2}$.

 Again if we consider  $T=\left(\begin{array}{cc}
0&1 \\
0&0
\end{array}\right)$ then $  \left(\frac{1}{2}w^2(T^2)+\frac{1}{8}\|T^{*}T+TT^{*}\|^2\right)^{\frac{1}{4}} = \sqrt{\sqrt {\frac{1}{8}}}$,  whereas   $   \frac{1}{2}\left(\|T\|+\min_{t\in [0,1]}w(\widetilde{T_t})\right) =    \frac{1}{2}.$
This shows that the upper bounds obtained in (\ref{number4}) and Theorem \ref{theorem:upper7} are not comparable. We observe that inequality (\ref{number4}) is sharper than (\ref{number1}) and the inequality obtained in Theorem \ref{theorem:upper7} is sharper than (\ref{number2}). Similarly using the same matrices one can conclude that inequality (\ref{number2}) is not comparable with inequality (\ref{number3}) and (\ref{number4}). 
\end{example}




\begin{thebibliography}{20}

\bibitem{AF2} A. Abu-Omar and F. Kittaneh, A numerical radius inequality involving the generalized Aluthge transform, Studia Math. 216(1) (2013) 69-75.

\bibitem{BBP} P. Bhunia, S. Bag and K. Paul, Numerical radius inequalities and its applications in estimation of zeros of polynomials, Linear Algebra Appl. 573 (2019) 166-177.

\bibitem{CM} F. Chabbabi and M. Mbekhta, New formulas for the spectral radius via $\lambda$-Aluthge transform, Linear Algebra Appl. 515 (2017) 246-254. 


\bibitem{GW} H.-L. Gau, and P.Y. Wu, Equality of three numerical radius inequalities, Linear Algebra Appl. 554 (2018) 51-67. 

\bibitem{H} E. Heinz, Beitr\"age zur St\"orungstheorie der Spektralzerlegung, Math. Ann 123 (1951) 415-438.

\bibitem{JKP} I.B. Jung, E. Ko and C. Pearcy, Aluthge transforms of operators, Integral Equations Operator Theory 37 (2000) 437-448.

\bibitem{FK2} F. Kittaneh, Numerical radius inequalities for Hilbert spaces operators, Studia Math. 168(1) (2005) 73-80. 

\bibitem{FK1} F. Kittaneh, Numerical radius inequality and an estimate for the numerical radius of the Frobenius companion matrix, Studia Math. 158(1) (2003) 11-17.

\bibitem{FK3} F. Kittaneh, Commutator inequalities associated with the polar decomposition, Proc. Amer. Math. Soc. 130 (2002) 1279-1283. 



\bibitem{O} K. Okubo,  On weakly unitarily invariant norm and the Aluthge transformation, Linear Algebra Appl. 371 (2003) 369-375.

\bibitem {PB} K. Paul and S. Bag, On the numerical radius of a matrix and estimation of bounds for zeros of a polynomial, Int. J. Math. Math. Sci. 2012 (2012) Article Id 129132, {https://doi.org/10.1155/2012/129132}.


\bibitem{Y} T. Yamazaki, On upper and lower bounds for the numerical radius and an equality condition, Studia Math. 178(1) (2007) 83-89.


\bibitem{ZFW} X. Zhou, J. Fang and S. Wen, A note on the $C$-numerical radius and the $\lambda$-Aluthge transform in finite factors, Ann. Funct. Anal. 9(4) (2018) 463-473.


\end{thebibliography}
\end{document}